\documentclass[preprint,12pt]{elsarticle}
\biboptions{numbers,sort&compress}

\usepackage{amssymb}
\usepackage{amsthm}
\usepackage{amsmath}
\usepackage{epic}
\usepackage{CJK}
\usepackage{setspace}
\usepackage{bm}
\newtheorem{theorem}{Theorem}[section]

\newtheorem{lemma}[theorem]{Lemma}

\newtheorem{definition}[theorem]{Definition}

\newtheorem{remark}{Remark}

\journal{}

\begin{document}
\begin{spacing}{1.15}
\begin{CJK*}{GBK}{song}
\begin{frontmatter}
\title{\textbf{An Erd\H{o}s-Stone type result for high-order spectra of graphs}}

\author[label1]{Chunmeng Liu}\ead{liuchunmeng0214@126.com}
\author[label1]{Jiang Zhou\corref{cor}}\ead{zhoujiang@hrbeu.edu.cn}
\author[label1]{Changjiang Bu}\ead{buchangjiang@hrbeu.edu.cn}
\cortext[cor]{Corresponding author}

\address{
\address[label1]{College of  Mathematical Sciences, Harbin Engineering University, Harbin 150001, PR China}
}

\begin{abstract}
Erd\H{o}s-Stone Theorem is a well-known result in extremal graph theory which determines the asymptotic behaviour of maximum number of edges in an $n$-vertex $H$-free graph. 
In 2009, Nikiforov gave a spectral version of Erd\H{o}s-Stone Theorem. 
In this paper, we obtain a tensor's spectral version of Erd\H{o}s-Stone Theorem. 
\end{abstract}

\end{frontmatter}

\section{Introduction}

The graphs in this paper are undirected and simple.
Let $G$ be a graph with the set of vertices $V(G)=\{1,2,...,n\}$.
If an induced subgraph of a subset of $V(G)$ is a complete graph, then the subset is called a clique. 
A clique is called a $t$-clique if it has $t$ elements. 
Let $C_{t}(G)$ denote the set of $t$-cliques in $G$. 
Denote by $c_{t}(G)$ the number of elements in $C_{t}(G)$.
A complete $s$-partite graph with the absolute value of the difference between the number of vertices in any two parts is at most $1$ is denote by $T_{s}(n)$, where $n$ is its number of vertices. 
And $T_{s}(n)$ is called the $s$-partite Tur\'{a}n graph.

For a graph $H$ with no isolated vertices, denote by $ex(n,H)$ the maximum number of edges in an $n$-vertex $H$-free graph. 
A celebrated result in extremal graph theory is called Erd\H{o}s-Stone Theorem \cite{Stone} (see also \cite{Simonovits}). 
\begin{theorem}\textup{\cite{Stone}} \label{thm stone-sim.}
Let $H$ be a graph with chromatic number $\chi(H)=r+1$. 
For an arbitrary positive number $\epsilon$, there is a positive number $n_{0}(\epsilon, H)$ such that for $n > n_{0}(\epsilon, H)$, 
\begin{align*}
\left(1-\frac{1}{r}-\epsilon\right)\frac{n^2}{2}\leq ex(n,H)\leq\left(1-\frac{1}{r}+\epsilon\right)\frac{n^2}{2}.
\end{align*}
\end{theorem}

For two graphs $T$ and $H$ with no isolated vertices, denote by $ex(n,T,H)$ the maximum number of copies of $T$ in an $n$-vertex $H$-free graph. 
For more results on determining $ex(n,T,H)$, one can refer to \cite{,Alon,Luo,Ergemlidze,Gerbner,Ma}.
An edge is a complete graph with two vertices. 
Alon and Shikhelman \cite{Alon} extended Theorem \ref{thm stone-sim.} to $K_{t}$ and obtained the following conclusion.
\begin{theorem}\textup{\cite{Alon}} \label{cor_Alon}
Let $H$ be a graph with chromatic number $\chi(H)=r+1$ and $r+1>t$. 
For an arbitrary positive number $\epsilon$, there is a positive number $n_{0}(\epsilon, H)$ such that for $n > n_{0}(\epsilon, H)$, 
\begin{align*}
\left(\binom{r}{t}\left(\frac{1}{r}\right)^{t}-\epsilon\right)n^{t}\leq ex(n,K_{t},H)\leq\left(\binom{r}{t}\left(\frac{1}{r}\right)^{t}+\epsilon\right)n^{t}.
\end{align*}
\end{theorem}

Let $ex_{\mu}(n,H)$ be the largest spectral radius of the adjacency matrix of an $H$-free $n$-vertex graph. 
There have been many results on determining $ex_{\mu}(n,H)$ (see \cite{Nikiforov_sur,Li,Tait,Zhai,Kang,Nikiforov}). 
A conclusion in \cite{Nikiforov} can derive the adjacency matrix's spectral version of Erd\H{o}s-Stone Theorem.
\begin{theorem}\textup{\cite{Nikiforov}} \label{thm Niki.}
Let $H$ be a graph with chromatic number $\chi(H)=r+1$. 
For an arbitrary positive number $\epsilon$, there is a positive number $n_{0}(\epsilon, H)$ such that for $n > n_{0}(\epsilon, H)$, 
\begin{align*}
\left(1-\frac{1}{r}-\epsilon\right)n\leq ex_{\mu}(n,H)\leq\left(1-\frac{1}{r}+\epsilon\right)n.
\end{align*}
\end{theorem}
\begin{remark}
For a graph $G$, from the inequality $\frac{2c_{2}(G)}{n}\leq\mu(G)$, where $c_{2}(G)$ is the number of edges of $G$, it is easy to see that Theorem \ref{thm Niki.} implies Theorem \ref{thm stone-sim.}. (see \cite{Nikiforov})
\end{remark}
 
In 1980, Cvetkovi$\acute{\mathrm{c}}$, Doob and Sachs \cite{Book1} posed the following high order spectrum of a graph $G$.
For a vertex $i\in V(G)$, the neighborhood of $i$ is the set of all vertices adjacent to $i$, denoted by $N(i)$. 
If there exists a complex number $\lambda$ and a nonzero complex vector $x=(x_{1},\ldots,x_{n})^{T}$ such that
\begin{align*}
\lambda x_{i}^{2}=
\begin{cases}
\sum\limits_{\substack{\{i_{2},i_{3}\}\subseteq N(i),\\ 
i_{2}\neq i_{3}}} x_{i_{2}}x_{i_{3}},& |N(i)|\geq 2,\\
0,& |N(i)|<2,
\end{cases} \ (i=1,2,\cdots,n)
\end{align*}
then $\lambda$ is called a quadratic eigenvalue of $G$ (see \cite{Book1}). 
In 2005, Qi \cite{Qi} and Lim \cite{Lim} posed the concept of eigenvalues of tensors. 
In fact, the quadratic eigenvalues of $G$ are the eigenvalues of special third-order tensors (the tensors and the eigenvalues of tensors will be introduced in Section 2).
In \cite{Liu_2}, the authors extended the quadratic spectrum of $G$ to a higher order and rewrote the high order spectrum of a graph $G$ to the spectrum of a tensor that comes from $G$.
Furthermore, the authors \cite{Liu} proposed another type of high order spectra of a graph through a clique tensor of a graph, and the spectral Mantel's Theorem is extended.
In this paper, we focus on spectra of the clique tensor. 

\begin{definition}\textup{\cite{Liu}}\label{defi cliques-tensor}
Let $G$ be a graph with $n$ vertices. An order $t$ dimension $n$ tensor $\mathcal{A}(G)=(a_{i_{1}i_{2}\cdots i_{t}})$ is called the $t$-clique tensor of $G$, where
\begin{align*}
a_{i_{1}i_{2}\cdots i_{t}}=
\begin{cases}
   \frac{1}{(t-1)!}, &\{i_{1},\ldots,i_{t}\}\in C_{t}(G).\\
   0, &\textup{otherwise}.
\end{cases}
\end{align*}
\end{definition}
The spectral radius of $\mathcal{A}(G)$ is called the $t$-clique spectral radius of $G$, denoted by $\mu^{(t)}(G)$. 
We write $ex_{\mu^{(t)}}(n,H)$ for the largest $t$-clique spectral radius of an $n$-vertex $H$-free graph. 
We have the following conclusion.
\begin{theorem} \label{thm main}
Let $H$ be a graph with chromatic number $\chi(H)=r+1$. 
For an arbitrary positive number $\epsilon$, there is a positive number $n_{0}(\epsilon, H)$ such that for $n > n_{0}(\epsilon, H)$, 
\begin{align*}
\left(\frac{1}{r^{r-1}}-\epsilon\right)n^{r-1}\leq ex_{\mu^{(r)}}(n,H)\leq\left(\frac{1}{r^{r-1}}+\epsilon\right)n^{r-1}.
\end{align*}
\end{theorem}

For Theorem \ref{cor_Alon}, when $t=r$, we have
\begin{align}\label{ineq Alon}
\left(\frac{1}{r^{r}}-\epsilon\right)n^{r}\leq ex(n,K_{r},H)\leq\left(\frac{1}{r^{r}}+\epsilon\right)n^{r}.
\end{align}
In Section 3, we will show that Theorem \ref{thm main} implies (\ref{ineq Alon}).

\section{Preliminaries}

We introduce some definitions and lemmas required for proofs in this section. 
An order $m$ dimension $n$ complex tensor $\mathcal{A}=(a_{i_{1}i_{2}...i_{m}})$ is a multidimensional array with $n^{m}$ entries, where $i_{j}=1,2,...,n$, $j=1,2,...,m$. 
Let $\mathbb{C}$ be the complex field and let $\mathbb{C}^{n}$ be the set of $n$-dimensional complex vectors.
For vectors $x=(x_{1},...,x_{n})^{T}\in\mathbb{C}^{n}$, the $\mathcal{A}x^{m-1}$ is a vector in $\mathbb{C}^{n}$ whose $i$-th component is $\sum_{i_{2},...,i_{m}=1}^{n}a_{ii_{2}...i_{m}}x_{i_{2}}\cdots x_{i_{m}}$ (see \cite{Qi}). 
If there exist $\lambda\in\mathbb{C}$ and nonzero vector $x=(x_{1},...,x_{n})^{T}\in\mathbb{C}^{n}$ satisfying
\begin{align*}
\mathcal{A}x^{m-1}=\lambda x^{[m-1]},
\end{align*}
then $\lambda$ is called an \emph{eigenvalue} of $\mathcal{A}$, and $x$ is called an \emph{eigenvector} of $\mathcal{A}$ corresponding to $\lambda$, where $x^{[m-1]}=(x_{1}^{m-1},...,x_{n}^{m-1})^{T}$. 
The largest modulus of all eigenvalues of $\mathcal{A}$ is called the spectral radius of $\mathcal{A}$.
A tensor is called nonnegative if its each entry is nonnegative. 
A tensor is called symmetric if its entries are invariant under any permutation of their indices.

\begin{lemma} \label{lem Qi} \textup{\cite{Qi2013}}
Suppose that $\mathcal{A}$ is an order $m$ dimension $n$ symmetric nonnegative tensor, with $m\geq2$. Then the spectral radius of $\mathcal{A}$ is equal to
\begin{align*}
\max\left\{x^{T}\mathcal{A}x^{m-1}: \sum_{i=1}^{n}x_{i}^{m}=1, x=(x_{1},\ldots,x_{n})^{T}\in\mathbb{R}_{+}^{n}\right\},
\end{align*}
where $\mathbb{R}_{+}^{n}$ is the set of $n$-dimensional nonnegative real vectors. 
And $\mathcal{A}$ has a nonnegative eigenvector corresponding to the spectral radius.
\end{lemma}

H\"{o}lder's inequality and Maclaurin's inequality are introduced in \cite{Hardy}.

\begin{lemma} \label{ineq. Holder} \textup{\cite{Hardy}} (H\"{o}lder's inequality)
Let $x=(x_{1},\ldots,x_{n})^{T}$ and $y=(y_{1},\ldots,y_{n})^{T}$ be nonnegative vectors. If the positive numbers $p$ and $q$ satisfy $\frac{1}{p}+\frac{1}{q}=1$, then
\begin{align*}
\sum_{i=1}^{n}x_{i}y_{i}\leq\left(\sum_{i=1}^{n}x_{i}^{p}\right)^{\frac{1}{p}}\left(\sum_{i=1}^{n}y_{i}^{q}\right)^{\frac{1}{q}}
\end{align*}
with equality holding if and only if $x$ and $y$ are proportional.
\end{lemma}

\begin{lemma} \label{ineq. Maclaurin} \textup{\cite{Hardy}} (Maclaurin's inequality)
Let $x=(x_{1},\ldots,x_{n})^{T}$ be a nonnegative vector. Then
\begin{align*}
\frac{x_{1}+\cdots+x_{n}}{n}\geq\left(\frac{\sum\limits_{1\leq i_{1}<\cdots<i_{k}\leq n}x_{i_{1}}\cdots x_{i_{k}}}{\binom{n}{k}}\right)^{\frac{1}{k}}
\end{align*}
for each $k\in\{1,2,\ldots,n\}$, the equality holds if and only if $x_{1}=\cdots=x_{n}$.
\end{lemma}

\begin{lemma}\textup{\cite{Liu}}\label{lem inequation of clique}
For a graph $G$ on $n$ vertices, we have
\begin{align}\label{equ p-spectral radio 1}
c_{t}(G)\leq\frac{n}{t}\mu^{(t)}(G).
\end{align}
Furthermore, if the number of $t$-cliques containing $i$ are equal for all $i\in V(G)$, then equality holds in \textup{(\ref{equ p-spectral radio 1})}.
\end{lemma}

\begin{lemma}\textup{\cite{Liu}}\label{lem main of 2th}
Let $G$ be a graph with $n$ vertices. 
If $G$ is $K_{r+1}$-free, then
\begin{align*}
\mu^{(r)}(G)\leq(n_{1}n_{2}\cdots n_{r})^{\frac{r-1}{r}}
\end{align*}
with equality holding if and only if $G\cong T_{r}(n)$, 
where $n_{s}$ is the number of vertices in $s$th part of $T_{r}(n)$ for $s=1,2,\ldots,r$.
\end{lemma}

\begin{lemma} \label{Erdos} \textup{\cite{Erdos}}
Let $\epsilon$ be an arbitrary positive number, and let $G$ be an $H$-free graph on $n$ vertices. 
There exists a positive number  $n_{0}(\epsilon, H)$ such that for $n > n_{0}(\epsilon, H)$, one can remove less than $\epsilon n^{2}$ edges from $G$ so that the remaining graph is $K_{r+1}$-free, where $r+1=\chi(H)$.
\end{lemma}

\section{Proofs of theorems}
In order to prove Theorem \ref{thm main}, we first give a lemma that provides an inequality between the number of $t$-cliques in a graph $G$ and the $t$-clique spectral radius. 
It is a generalization of Wilf's \cite{Wilf} inequality $\mu(G)\leq \sqrt{2(1-1/n)c_{2}(G)}$, where $\mu(G)$ is the spectral radius of the adjacency matrix of $G$.
\begin{lemma} \label{ineq. t-clique spec.}
Let $G$ be a graph with $n$ vertices. Then
\begin{align*}
\mu^{(t)}(G)\leq \frac{t}{n}\binom{n}{t}^{\frac{1}{t}}\left(c_{t}(G)\right)^{\frac{t-1}{t}}.
\end{align*}
Futhermore, equality holds if $G$ is a complete graph or a $K_{t}$-free graph.
\end{lemma}
\begin{proof}
Let $\mathcal{A}(G)=(a_{i_{1}i_{2}\cdots i_{t}})$ be the $t$-clique tensor of $G$. 
Let a nonnegative vector $x=(x_{1},x_{2},\ldots,x_{n})^{T}$ be the eigenvector corresponding to the $t$-clique spectral radius $\mu^{(t)}(G)$ and $x_{1}^{t}+x_{2}^{t}+\cdots+x_{n}^{t}=1$. 
By Lemma \ref{lem Qi}, we obtain
\begin{align}\label{equ 1}
\mu^{(t)}(G)&=x^{T}\mathcal{A}(G)x^{t-1} \nonumber \\
&=\sum_{i_{1},i_{2},\ldots,i_{t}=1}^{n}a_{i_{1}i_{2}\cdots i_{t}}x_{i_{1}}x_{i_{2}}\cdots x_{i_{t}} 
=\sum_{\{i_{1},i_{2},\ldots,i_{t}\}\in C_{t}(G)}\frac{1}{(t-1)!}x_{i_{1}}x_{i_{2}}\cdots x_{i_{t}} 
\nonumber \\ 
&=t\sum_{\substack{\{i_{1},i_{2},\ldots,i_{t}\}\in C_{t}(G),\\i_{1}<i_{2}<\cdots<i_{t}}}x_{i_{1}}x_{i_{2}}\cdots x_{i_{t}}.
\end{align}
From (\ref{equ 1}) and Lemma \ref{ineq. Holder}, we get
\begin{align}\label{equ 2}
\mu^{(t)}(G)
&\leq t\left(\sum_{\substack{\{i_{1},i_{2},\ldots,i_{t}\}\in C_{t}(G),\\i_{1}<i_{2}<\cdots<i_{t}}}1^{\frac{t}{t-1}}\right)^{\frac{t-1}{t}}\left(\sum_{\substack{\{i_{1},i_{2},\ldots,i_{t}\}\in C_{t}(G),\\i_{1}<i_{2}<\cdots<i_{t}}}x_{i_{1}}^{t}x_{i_{2}}^{t}\cdots x_{i_{t}}^{t}\right)^{\frac{1}{t}} \nonumber\\ 
&= t\left(c_{t}(G)\right)^{\frac{t-1}{t}}\left(\sum_{\substack{\{i_{1},i_{2},\ldots,i_{t}\}\in C_{t}(G),\\i_{1}<i_{2}<\cdots<i_{t}}}x_{i_{1}}^{t}x_{i_{2}}^{t}\cdots x_{i_{t}}^{t}\right)^{\frac{1}{t}}.
\end{align}
By Lemma \ref{ineq. Maclaurin}, we have
\begin{align}\label{equ 4}
\left(\sum_{\substack{\{i_{1},i_{2},\ldots,i_{t}\}\in C_{t}(G),\\i_{1}<i_{2}<\cdots<i_{t}}}x_{i_{1}}^{t}x_{i_{2}}^{t}\cdots x_{i_{t}}^{t}\right)^{\frac{1}{t}}
&\leq\left(\sum_{1\leq i_{1}<i_{2}<\cdots<i_{t}\leq n}x_{i_{1}}^{t}x_{i_{2}}^{t}\cdots x_{i_{t}}^{t}\right)^{\frac{1}{t}} \nonumber \\
&\leq\left(\frac{x_{1}^{t}+x_{2}^{t}+\cdots+x_{n}^{t}}{n}\right)\binom{n}{t}^{\frac{1}{t}}.
\end{align}
Combining inequalities (\ref{equ 2}) and (\ref{equ 4}), we obtain
\begin{align*}
\mu^{(t)}(G)\leq \frac{t}{n}\binom{n}{t}^{\frac{1}{t}}\left(c_{t}(G)\right)^{\frac{t-1}{t}}.
\end{align*}

If $G$ is a  $K_{t}$-free graph, we get $c_{t}(G)=0$ and the $t$-clique tensor of $G$ is a zero tensor implies that $\mu^{(t)}(G)=0$. 
If $G$ is a complete graph, the number of $t$-cliques containing each vertex in $V(G)$ is equal, thus $\mu^{(t)}(G)=\frac{t}{n}\binom{n}{t}$ by Lemma \ref{lem inequation of clique}. 
Therefore, equality holds if $G$ is a complete graph or a $K_{t}$-free graph.
\end{proof}

Next, we give the proof of Theorem \ref{thm main}.
\begin{proof}[Proof of Theorem \ref{thm main}]
Let $H$ be a graph with chromatic number $\chi(H)=r+1$. 
Suppose that the number of vertices in $s$th part of the $r$-partite Tur\'{a}n graph $T_{r}(n)$ is $n_{s}$ for $s=1,2,\ldots,r$. 
Since the chromatic number of $T_{r}(n)$ is equal to $r$, the graph $T_{r}(n)$ is $H$-free. 
Thus
\begin{align*}
ex_{\mu^{(r)}}(n,H)\geq \mu^{(r)}(T_{r}(n)).
\end{align*}
For the graph $T_{r}(n)$, by Lemma \ref{lem inequation of clique}, we have
\begin{align*}
\mu^{(r)}(T_{r}(n))\geq\frac{rc_{r}(T_{r}(n))}{n}.
\end{align*}
The number of $r$-cliques in $T_{r}(n)$ is equal to $n_{1}n_{2}\cdots n_{r}$,  where $n_{s}$ is equal to $\lfloor\frac{n}{r}\rfloor$ or $\lceil\frac{n}{r}\rceil$ for $s=1,2,\ldots,r$. 
Then
\begin{align*}
\frac{rc_{r}(T_{r}(n))}{n}=\frac{r}{n}\left(n_{1}n_{2}\cdots n_{r}\right)\geq\frac{r}{n}\left(\frac{n}{r}-1\right)\cdots \left(\frac{n}{r}-1\right)
=\frac{r}{n}\left(\frac{n}{r}-1\right)^{r}.
\end{align*}
For an arbitrary positive number $\epsilon$, let $n_{0}^{\prime}(\epsilon,H)=\frac{2^{r-1}}{\epsilon r^{r-2}}$.
For $n > n_{0}^{\prime}(\epsilon,H)$, we have
\begin{align*}
ex_{\mu^{(r)}}(n,H)
&\geq\frac{r}{n}\left(\frac{n}{r}-1\right)^{r}=\frac{r}{n}\left(\sum_{k=0}^{r}\binom{r}{k}\left(\frac{n}{r}\right)^{r-k}(-1)^{k}\right)
\geq\frac{r}{n}\left(\frac{n^{r}}{r^{r}}-2^{r-1}\frac{n^{r-1}}{r^{r-1}}\right)\\
&=\left(\frac{1}{r^{r-1}}-\frac{2^{r-1}}{r^{r-2}}\frac{1}{n}\right)n^{r-1}\geq\left(\frac{1}{r^{r-1}}-\epsilon\right)n^{r-1}.
\end{align*}

On the other hand, for an arbitrary positive number $\epsilon$, let the positive number $\epsilon^{\prime}\in \left(0,\left(\frac{\epsilon}{e}\right)^{\frac{r}{r-1}}\frac{1}{e^{r-2}}\right]$.
Since $G$ is $H$-free, by Lemma \ref{Erdos}, there exists a positive number $n_{0}^{\prime\prime}(\epsilon^{\prime}, H)$ such that for $n > n_{0}^{\prime\prime}(\epsilon^{\prime}, H)$, 
one remove $\epsilon^{\prime} n^{2}$ edges from $G$ so that the remaining graph $G^{\prime}$ is $K_{r+1}$-free. 
Each edge of $G$ is contained in at most $\binom{n-2}{r-2}$ copies of $K_{r}$ which implies that we remove at most $\epsilon^{\prime} n^{2}\binom{n-2}{r-2}$ copies of $K_{r}$ from $G$. 
Let $\mathcal{A}(G)$ and $\mathcal{A}(G^{\prime})$ be the $r$-clique tensors of $G$ and $G^{\prime}$, respectively.
Let $x$ be the nonnegative eigenvector of $\mathcal{A}(G)$ corresponding to $\mu^{(r)}(G)$. 
Thus, we have 
\begin{align*}
\mu^{(r)}(G) = x^{T}\mathcal{A}(G)x^{r-1}=x^{T}\mathcal{A}(G^{\prime})x^{r-1}+x^{T}(\mathcal{A}(G)-\mathcal{A}(G^{\prime}))x^{r-1}.
\end{align*}
Let $\mathcal{G}$ be the set of subgraphs of $G$, where the $t$-clique tensor of each graph in $\mathcal{G}$ is the tensor $\mathcal{A}(G)-\mathcal{A}(G^{\prime})$. 
The graphs in $\mathcal{G}$ have the same number of $r$-cliques.
There is a graph $G^{\prime\prime}$ corresponding to the tensor $\mathcal{A}(G)-\mathcal{A}(G^{\prime})$. 
By Lemma \ref{lem Qi}, we get 
\begin{align*}
\mu^{(r)}(G)\leq \mu^{(r)}(G^{\prime})+\mu^{(r)}(G^{\prime\prime}).
\end{align*}
Since we remove at most $\epsilon^{\prime} n^{2}\binom{n-2}{r-2}$ copies of $K_{r}$ from $G$, the $r$-cliques in $G^{\prime\prime}$ is less than or equal to $\epsilon^{\prime} n^{2}\binom{n-2}{r-2}$.
By Lemma \ref{ineq. t-clique spec.}, we obtain 
\begin{align}\label{ieq. myself}
\mu^{(r)}(G^{\prime\prime})
&\leq \frac{r}{n}\binom{n}{r}^{\frac{1}{r}}\left(c_{r}(G^{\prime\prime})\right)^{\frac{r-1}{r}}
\leq\frac{r}{n}\binom{n}{r}^{\frac{1}{r}}\left(\epsilon^{\prime} n^{2}\binom{n-2}{r-2}\right)^{\frac{r-1}{r}}\nonumber\\
&\leq\frac{r}{n}\frac{en}{r}\left(\epsilon^{\prime}n^{2}\left(\frac{en}{r-2}\right)^{r-2}\right)^{\frac{r-1}{r}}
=e\left(\epsilon^{\prime}\left(\frac{e}{r-2}\right)^{r-2}\right)^{\frac{r-1}{r}}n^{r-1}\nonumber\\
&\leq e\left(\epsilon^{\prime}e^{r-2}\right)^{\frac{r-1}{r}}n^{r-1}
\leq\epsilon n^{r-1},
\end{align}
the third inequality in (\ref{ieq. myself}) follows from $\binom{n}{r}\leq\left(\frac{en}{r}\right)^{r}$, where $\ln e=1$.
The graph $G^{\prime}$ is $K_{r+1}$-free, by Lemma \ref{lem main of 2th}, we have
\begin{align*}
\mu^{(r)}(G^{\prime})\leq\mu^{(r)}(T_{r}(n)).
\end{align*}
For the graph $T_{r}(n)$, by Lemma \ref{lem main of 2th}, it is shown that $\mu^{(r)}(T_{r}(n))=(n_{1}n_{2}\cdots n_{r})^{\frac{r-1}{r}}$.
By mean inequality, we get
\begin{align*}
\mu^{(r)}(T_{r}(n))=(n_{1}n_{2}\cdots n_{r})^{\frac{r-1}{r}}\leq\left(\frac{n_{1}+n_{2}+\cdots+n_{r}}{r}\right)^{r-1}=\left(\frac{n}{r}\right)^{r-1}.
\end{align*}
Thus, we have
\begin{align*}
ex_{\mu^{(r)}}(n,H)\leq \left(\frac{n}{r}\right)^{r-1}+\epsilon n^{r-1}=\left(\frac{1}{r^{r-1}}+\epsilon\right)n^{r-1}.
\end{align*}
For an arbitrary positive number $\epsilon$, let $n_{0}(\epsilon, H)=\max\{n_{0}^{\prime}(\epsilon,H),n_{0}^{\prime\prime}(\epsilon^{\prime}, H)\}$. 
For $n > n_{0}(\epsilon, H)$, 
\begin{align*}
\left(\frac{1}{r^{r-1}}-\epsilon\right)n^{r-1}\leq ex_{\mu^{(r)}}(n,H)\leq\left(\frac{1}{r^{r-1}}+\epsilon\right)n^{r-1}.
\end{align*}
This completes the proof.
\end{proof}

In the following, we give the proof of (\ref{ineq Alon}) based on Theorem \ref{thm main}.
\begin{proof}[Proof of Theorem \ref{cor_Alon}]
Since the graph $T_{r}(n)$ is $H$-free, we get 
\begin{align*}
ex(n,K_{r},H)\geq c_{r}(T_{r}(n)).
\end{align*}
For an arbitrary positive number $\epsilon$, let $n_{0}^{\prime}(\epsilon,H)=\frac{2^{r-1}}{\epsilon r^{r-1}}$.
For $n > n_{0}^{\prime}(\epsilon,H)$, we have
\begin{align*}
c_{r}(T_{r}(n))
&\geq\left(\frac{n}{r}-1\right)^{r}
\geq\left(\frac{n^{r}}{r^{r}}-2^{r-1}\frac{n^{r-1}}{r^{r-1}}\right)
=\left(\frac{1}{r^{r}}-\frac{2^{r-1}}{r^{r-1}}\frac{1}{n}\right)n^{r}\\
&\geq\left(\frac{1}{r^{r}}-\epsilon\right)n^{r}.
\end{align*}
By Lemma \ref{lem inequation of clique}, we obtain
\begin{align*}
ex(n,K_{r},H)\leq\frac{n}{r}ex_{\mu^{(r)}}(n,H).
\end{align*}
For an arbitrary positive number $\epsilon$, let the positive number $\epsilon_{0}\in \left(0,r\epsilon\right]$. 
From Theorem \ref{thm main}, there is a positive number $n_{0}^{\prime\prime}(\epsilon_{0}, H)$ such that for $n > n_{0}^{\prime\prime}(\epsilon_{0}, H)$,
\begin{align*}
ex(n,K_{r},H)\leq\frac{n}{r}\left(\frac{1}{r^{r-1}}+\epsilon_{0}\right)n^{r-1}\leq\left(\frac{1}{r^{r}}+\epsilon\right)n^{r}.
\end{align*}
For an arbitrary positive number $\epsilon$, let $n_{0}(\epsilon, H)=\max\{n_{0}^{\prime}(\epsilon,H),n_{0}^{\prime\prime}(\epsilon_{0}, H)\}$. 
For $n > n_{0}(\epsilon, H)$, 
\begin{align*}
\left(\frac{1}{r^{r}}-\epsilon\right)n^{r}\leq ex(n,K_{r},H)\leq\left(\frac{1}{r^{r}}+\epsilon\right)n^{r}.
\end{align*}
The proof is completed.
\end{proof}

\section*{Acknowledgement}

This work is supported by the National Natural Science Foundation of China (No. 12071097, 12371344), the Natural Science Foundation for The Excellent Youth Scholars of the Heilongjiang Province (No. YQ2022A002) and the Fundamental Research Funds for the Central Universities.

\end{CJK*}
\end{spacing}
\end{document}